\documentclass[12pt,twoside]{amsart}
\usepackage{amsmath}
\usepackage{amsthm}
\usepackage{amsfonts}
\usepackage{amssymb}
\usepackage{latexsym}
\usepackage{mathrsfs}
\usepackage{amsmath}
\usepackage{amsthm}
\usepackage{amsfonts}
\usepackage{amssymb}
\usepackage{latexsym}
\usepackage{geometry}
\usepackage{dsfont}
\usepackage[dvips]{graphicx}
\usepackage{color}
\usepackage[all]{xy}

\date{}
\allowdisplaybreaks[4] \footskip=15pt
\renewcommand{\uppercasenonmath}[1]{}

\usepackage{graphicx,amssymb}
\usepackage[all]{xy}
\usepackage{amsmath}

\allowdisplaybreaks
\usepackage{amsthm}
\usepackage{color}

\theoremstyle{plain}
\newtheorem{theorem}{Theorem}[section]
\newtheorem{proposition}[theorem]{Proposition}
\newtheorem{lemma}[theorem]{Lemma}
\newtheorem{corollary}[theorem]{Corollary}
\newtheorem{example}[theorem]{Example}
\newtheorem*{open question}{Open Question}
\newtheorem{definition}[theorem]{Definition}

\theoremstyle{definition}
\newtheorem*{acknowledgement}{Acknowledgement}
\newtheorem*{theo}{Theorem}

\theoremstyle{remark}
\newtheorem{remark}[theorem]{Remark}

\newcommand{\Tor}{\mbox{\rm Tor}}

\def\p{\frak p}
\def\m{\frak m}

\def\Hom{{\rm Hom}}
\def\Ext{{\rm Ext}}
\def\Tor{{\rm Tor}}

\def\Ker{{\rm Ker}}

\def\Im{{\rm Im}}
\def\Coker{{\rm Coker}}

\def\Max{{\rm Max}}

\def\Spec{{\rm Spec}}
\def\Max{{\rm Max}}
\begin{document}
\begin{center}
{\large  \bf Uniformly $S$-Noetherian rings}

\vspace{0.5cm}   Wei Qi$^{a}$,\  Hwankoo Kim$^{b}$,\ Fanggui Wang$^{c}$,\ Mingzhao Chen$^{d}$,\ Wei Zhao$^{e}$

{\footnotesize a.\ School of Mathematics and Statistics, Shandong University of Technology, Zibo 255049, China\\
b.\ Division of Computer and Information Engineering, Hoseo University, Asan 31499, Republic of Korea\\
c.\  School of Mathematical Sciences, Sichuan Normal University, Chengdu 610068,  China\\
d.\ College of Mathematics and Information Science, Leshan Normal University, Leshan 614000, China\\
e.\ School of Mathematics, ABa Teachers University, Wenchuan 623002,  China

}
\end{center}

\bigskip
\centerline { \bf  Abstract}
\bigskip
\leftskip10truemm \rightskip10truemm \noindent

Let $R$ be a ring and $S$  a multiplicative subset of $R$. Then $R$ is called a uniformly $S$-Noetherian ($u$-$S$-Noetherian for abbreviation) ring  provided there exists an element $s\in S$  such that for any ideal $I$ of $R$, $sI \subseteq K$ for some finitely generated sub-ideal $K$ of $I$. We give the Eakin-Nagata-Formanek Theorem for $u$-$S$-Noetherian rings. Besides, the $u$-$S$-Noetherian properties on several  ring constructions are given. The notion of $u$-$S$-injective modules is also introduced and studied. Finally, we obtain the Cartan-Eilenberg-Bass Theorem for uniformly $S$-Noetherian rings.
\vbox to 0.3cm{}\\
{\it Key Words:} $u$-$S$-Noetherian rings, $S$-Noetherian rings, $u$-$S$-injective modules, ring constructions.\\
{\it 2010 Mathematics Subject Classification:}  13E05, 13A15.

\leftskip0truemm \rightskip0truemm
\bigskip

\section{Introduction}
Throughout this article, $R$ is always  a commutative ring with identity. For a subset $U$ of  an $R$-module $M$, we denote by $\langle U\rangle$ the submodule of $M$ generated by $U$. A subset $S$ of $R$ is called a multiplicative subset of $R$ if $1\in S$ and $s_1s_2\in S$ for any $s_1\in S$, $s_2\in S$.  Recall from  Anderson and Dumitrescu \cite{ad02} that a ring $R$ is called an \emph{$S$-Noetherian ring} if for any ideal $I$ of $R$, there is a finitely generated sub-ideal $K$ of $I$  such that $sI\subseteq K$ for some $s\in S$. Cohen's Theorem, Eakin-Nagata Theorem and Hilbert Basis Theorem for $S$-Noetherian rings are given in \cite{ad02}. Many algebraists  have paid considerable attention to the notion of $S$-Noetherian rings, especially in the $S$-Noetherian properties of ring constructions. In 2007, Liu \cite{l07} characterized a ring $R$ when the generalized power series ring $[[R^{M,\leq}]]$ is an $S$-Noetherian ring under some additional conditions.  In 2014, Lim and Oh \cite{lO14} obtained some $S$-Noetherian properties on amalgamated algebras along and ideal.  They \cite{lO15} also studied $S$-Noetherian properties on the composite semigroup rings and the composite generalized series rings next year. In 2016, Ahmed and Sana \cite{as16} gave an $S$-version of  Eakin-Nagata-Formanek Theorem for  $S$-Noetherian rings in the case where $S$ is finite. Very recently, Kim, Mahdou, and Zahir \cite{kmz21}  established a necessary and sufficient condition for a bi-amalgamation to inherit the $S$-Noetherian property. Some generalizations of $S$-Noetherian ring can be found in \cite{bh18,kkl14}.

However, in the definition of $S$-Noetherian rings, the choice of $s\in S$ such that $sI\subseteq K\subseteq I$ with $K$ finitely generated  is dependent on the ideal $I$. This dependence sets many obstacles to the further study of $S$-Noetherian rings. The main motivation of this article is to introduce and study a ``uniform''  version of  $S$-Noetherian rings. In fact, we say  a ring $R$ is \emph{uniformly $S$-Noetherian} ($u$-$S$-Noetherian for abbreviation) provided there exists an element $s\in S$ such that for any ideal $I$ of $R$, $sI \subseteq K$ for some finitely generated sub-ideal $K$ of $I$. Trivially, Noetherian rings are $u$-$S$-Noetherian, and $u$-$S$-Noetherian rings are $S$-Noetherian. Some counterexamples are given in Example \ref{exam-not-ut} and  Example \ref{exam-not-ut-1}. We also consider the notion of $u$-$S$-Noetherian  modules (see Definition \ref{us-no-module}), and then obtain the Eakin-Nagata-Formanek Theorem for $u$-$S$-Noetherian modules (see Theorem \ref{u-s-noe-char}) which generalizes some part of the result in \cite[Corollary 2.1]{as16}. The $S$-extension property of $S$-Noetherian modules  is given in Proposition \ref{s-u-noe-s-exact}. In Section $3$, we mainly consider  the $u$-$S$-Noetherian properties on some  ring constructions, including trivial extensions, pullbacks and amalgamated algebras along an ideal (see Proposition \ref{trivial extension-usn},  Proposition \ref{pullback-usn} and  Proposition \ref{amag-usn}). In Section $4$, we first introduce the notion of $u$-$S$-injective modules $E$ for which $\Hom_R(-,E)$ preserves $u$-$S$-exact sequences (see Definition \ref{u-S-tor-ext}), and then characterize  it by $u$-$S$-torsion properties of the ``Ext'' functor in Theorem \ref{s-inj-ext}. The Baer's Criterion for $u$-$S$-injective modules is given in Proposition \ref{s-inj-baer}. Finally, we obtain the Cartan-Eilenberg-Bass Theorem for uniformly $S$-Noetherian rings as follows (see Theorem \ref{s-injective-ext}):
\begin{theo}\label{s-injective-ext}
Let $R$ be a ring, $S$ a multiplicative subset of $R$ consisting of non-zero-divisors. Then the following assertions are equivalent:
\begin{enumerate}
\item  $R$ is  $u$-$S$-Noetherian;
\item  any direct sum of injective modules is $u$-$S$-injective;
\item  any direct union of injective modules is $u$-$S$-injective.
\end{enumerate}
\end{theo}

\section{$u$-$S$-Noetherian rings and $u$-$S$-Noetherian modules}
Let $R$ be a ring and $S$ a multiplicative subset of $R$. Recall from \cite{ad02} that  $R$ is called an $S$-Noetherian ring (resp., $S$-PIR) provided  that for any ideal $I$ of $R$ there exists an element $s\in S$ and a finitely (resp., principally) generated sub-ideal $K$ of $I$ such that $sI\subseteq K$. Note that the choice of $s$ is decided by the ideal $I$. Now we introduce some  ``uniform''  versions of $S$-Noetherian rings and $S$-PIRs.

\begin{definition} Let $R$ be a ring and $S$ a multiplicative subset of $R$.
\begin{enumerate}
\item $R$ is called a $u$-$S$-Noetherian ring  provided there exists an element $s\in S$ such that for any ideal $I$ of $R$, $sI \subseteq K$ for some finitely generated sub-ideal $K$ of $I$.
\item  $R$ is called a $u$-$S$-Principal ideal ring $($$u$-$S$-PIR for short$)$ provided there exists an element $s\in S$ such that  for any ideal $I$ of $R$, $sI \subseteq (a)$ for some element $a\in I$.
\end{enumerate}
\end{definition}

If the element $s$ can be chosen to be  the identity in the definition of $u$-$S$-Noetherian rings, then $u$-$S$-Noetherian rings are exactly  Noetherian rings. Thus, every Noetherian ring is $u$-$S$-Noetherian. However, the converse does not hold generally.
\begin{example}\label{exam-not-ut}
Let $R=\prod\limits_{i=1}^{\infty}\mathbb{Z}_2$ be the countable  infinite direct product of $\mathbb{Z}_2$, then $R$ is not  Noetherian. Let $e_i$ be the element in $R$ with the $i$-th component $1$ and others $0$. Denote $S=\{1,e_i|i\geq 1\}$. Then $R$ is $u$-$S$-Noetherian. Indeed, let $I$ be an ideal of $R$. Then if all elements in $I$ have $1$-th components equal to $0$, we have $e_1I=0$. Otherwise  $e_1I=e_1R$. Thus $e_1I$ is principally generated. Consequently $R$ is a $u$-$S$-PIR, and so is $u$-$S$-Noetherian.
\end{example}

Let $R$ be a ring, $M$ an $R$-module and $S$ a multiplicative subset of $R$. For any $s\in S$, there is a  multiplicative subset $S_s=\{1,s,s^2,....\}$ of $S$. We denote by $M_s$ the localization of $M$ at $S_s$. Certainly, $M_s\cong M\otimes_RR_s$

\begin{lemma} \label{s-loc-u-noe}
 Let $R$ be a ring and $S$ a multiplicative subset of $R$. If $R$ is a $u$-$S$-Noetherian ring $($resp., $u$-$S$-PIR$)$, then there exists an element $s\in S$ such that $R_{s}$ is a Noetherian ring $($resp., PIR$)$.
\end{lemma}
\begin{proof} Since $R$ is $u$-$S$-Noetherian, there exists an element $s\in S$ satisfies that for any ideal $I$ of $R$ there is a finitely (resp., principally) generated sub-ideal $K$ of $I$ such that $sI\subseteq K$. Let $J$ be an ideal of  $R_{s}$. Then there exists an ideal $I'$ of $R$ such that $J= I'_s$, and hence $sI' \subseteq K'$ for some finitely (resp., principally) generated sub-ideal $K'$ of $I'$. So  $J= I'_s= K_s$ is finitely (resp., principally) generated ideal of $R_{s}$. Consequently, $R_{s}$ is a Noetherian  ring (resp., PIR).
\end{proof}

\begin{proposition} \label{s-loc-u-noe-fini}
Let $R$ be a ring and $S$ a multiplicative subset of $R$ consisting of finite  elements. Then $R$ is a $u$-$S$-Noetherian ring $($resp., $u$-$S$-PIR$)$ if and only if $R$ is an  $S$-Noetherian ring  $($resp.,  $S$-PIR$)$.
\end{proposition}
\begin{proof}  Suppose  $R$ is a $u$-$S$-Noetherian ring $($resp., $u$-$S$-PIR$)$. Then trivially $R$ is an $S$-Noetherian ring (resp., $S$-PIR). Suppose $S=\{s_1,...,s_n\}$ and set $s=s_1\cdots s_n$. Suppose $R$ is an $S$-Noetherian ring (resp., $S$-PIR). Then for any ideal $I$ of $R$, there is a finitely (resp., principally) generated sub-ideal $J$ of $I$
such that $s_II\subseteq  J$ for some $s_I\in S$. Then $sI\subseteq s_II\subseteq J$. So $R$ is a $u$-$S$-Noetherian ring $($resp., $u$-$S$-PIR$)$.
\end{proof}
The following example shows $S$-Noetherian rings are not  $u$-$S$-Noetherian in general.
\begin{example}\label{exam-not-ut-1}
Let $R=k[x_1,x_2,....]$ be the countably infinite variables polynomial ring over a field $k$. Set $S=R-\{0\}$. Then $R$ is an  $S$-Noetherian ring. However, $R$ is not $u$-$S$-Noetherian.
\end{example}
\begin{proof}
Certainly, $R$ is an $S$-Noetherian ring. Indeed,  let $I$ be a non-zero ideal of $R$. Suppose $0\not=s\in I$. Then $sI\subseteq sR\subseteq I$. Thus $I$ is $S$-principally generated. So $R$ is an $S$-PIR and thus an  $S$-Noetherian ring.

We claim that $R$ is not $u$-$S$-Noetherian. Assume on the contrary that $R$ is $u$-$S$-Noetherian. Then $R_{s}$ is a Noetherian ring for some $s\in S$ by Lemma \ref{s-loc-u-noe}. If $n$ is the minimal number such that $x_m$ does not divide any monomial of $s$ for any $m\geq n$. Then $R_{s}\cong T[x_n,x_{n+1},....]$ where $T=k[x_1,x_2,....,x_{n-1}]_s$.  Obviously, $R_{s}\cong T[x_n,x_{n+1},....]$ is not Noetherian since the ideal generated by $\{x_n,x_{n+1},....\}$ is not a finitely generated ideal of  $T[x_n,x_{n+1},....]$.  So $R$ is not $u$-$S$-Noetherian.
\end{proof}

Recall from \cite{ad02} that an $R$-module $M$ is called an $S$-Noetherian module if every submodule of $M$ is $S$-finite, that is, for any submodule $N$ of $M$, there is an element $s\in S$ and a finitely generated $R$-module $F$ such that $sN\subseteq F\subseteq N$. Note that the choice of $s$ is decided by the submodule $N$. The rest of this section mainly studies a ``uniform'' version of $S$-Noetherian modules. Let $\{M_j\}_{j\in \Gamma}$  be a family of $R$-modules and $N_j$ a submodule of $M_j$ generated by $\{m_{i,j}\}_{i\in \Lambda_j}\subseteq M_j$ for each $j\in \Gamma$. Recall from \cite{z21} that  a family of $R$-modules  $\{M_j\}_{j\in \Gamma}$  is \emph{$u$-$S$-generated} (with respective to $s$) by  $\{\{m_{i,j}\}_{i\in \Lambda_j}\}_{j\in \Gamma}$ provided that there exists an element $s\in S$  such that $sM_j\subseteq N_j$ for each $j\in \Gamma$, where $N_j=\langle \{m_{i,j}\}_{i\in \Lambda_j}\rangle$.  We say a family of $R$-modules  $\{M_j\}_{j\in \Gamma}$ is \emph{$u$-$S$-finite} (with respective to $s$) if the set $\{m_{i,j}\}_{i\in \Lambda_j}$ can be chosen as a finite set for each $j\in \Gamma$.

\begin{definition}\label{us-no-module} Let $R$ be a ring and $S$ a multiplicative subset of $R$. An $R$-module $M$ is called a $u$-$S$-Noetherian $R$-module provided the set of all submodules of $M$ is $u$-$S$-finite.
\end{definition}

Let $R$ be a ring and $S$ a multiplicative subset of $R$. Recall from \cite{z21}, an $R$-module $T$ is called a \emph{$u$-$S$-torsion module}  provided that there exists an element $s\in S$ such that $sT=0$. Obviously,  $u$-$S$-torsion modules are $u$-$S$-Noetherian. A ring $R$ is $u$-$S$-Noetherian  if and only if it is $u$-$S$-Noetherian as an $R$-module. It is well known that an $R$-module $M$ is Noetherian if and only if $M$ satisfies ascending chain condition on submodules, if and only if $M$ satisfies the maximal condition (see \cite{n93}). In 2016, Ahmed et al. \cite{as16} obtained an $S$-version of this result provided  $S$ is a finite set and called it the $S$-version of Eakin-Nagata-Formanek Theorem. Next we will give a uniformly $S$-version of Eakin-Nagata-Formanek Theorem for any multiplicative subset $S$ of $R$.

First,  we recall from \cite[Definition 2.1]{as16} some modified notions of $S$-stationary  ascending chains of $R$-modules and  $S$-maximal elements of a family  of  $R$-modules. Let $R$ be a ring, $S$ a multiplicative subset of $R$  and $M$ an $R$-module. Denote by $M^{\bullet}$  an  ascending chain $M_1\subseteq M_2\subseteq ... $ of submodules of $M$. An ascending chain $M^{\bullet}$ is called \emph{stationary with respective to $s$} if there exists $k\geq 1$ such that $sM_n\subseteq M_k$ for any $n\geq k$.   Let $\{M_i\}_{i\in \Lambda}$ be a family of submodules of $M$. We say an $R$-module  $M_0\in \{M_i\}_{i\in \Lambda}$ is \emph{maximal with respective to $s$}  provided that if  $M_0\subseteq M_i$ for some $M_i\in \{M_i\}_{i\in \Lambda}$, then $sM_i\subseteq M_0$.

\begin{theorem} \label{u-s-noe-char} {\bf (Eakin-Nagata-Formanek Theorem for uniformly $S$-Noetherian rings)}
Let $R$ be a ring and $S$ a multiplicative subset of $R$. Let $M$ be an $R$-module.  Then the following assertions are equivalent:
\begin{enumerate}
\item $M$ is $u$-$S$-Noetherian;
\item there exists an element $s\in S$ such that any  ascending chain of submodules of $M$ is stationary with respective to $s$;

\item  there exists an element $s\in S$ such that any non-empty of submodules of $M$ has a maximal element with respective to $s$.
\end{enumerate}
\end{theorem}
\begin{proof} $(1)\Rightarrow (2):$ Let $M_1\subseteq M_2\subseteq ... $  be an  ascending chain of submodules of $M$. Set $M_0=\bigcup\limits_{i=1}^{\infty}M_i$. Then there exist an element $s\in S$ and a finitely generated submodule $N_i$ of $M_i$ such that $sM_i\subseteq N_i$ for each $i\geq 0$. 
Since $N_0$ is finitely generated, there exists $k\geq 1$ such that $N_0\subseteq M_k$. Thus $sM_0\subseteq M_k$. So $sM_n\subseteq M_k$ for any $n\geq k$.

$(2)\Rightarrow (3):$ Let $\Gamma$ be a nonempty of submodules of $M$. On the contrary,  we take any $M_1\in \Gamma$. Then $M_1$ is not a maximal element with respective to $s$ for any $s\in S$. Thus there is  $M_2\in \Gamma$ such that $s M_2\not\subseteq M_1$. Since $M_2$ is not a maximal element with respective to $s$, there is $M_3\in \Gamma$, such that $s M_3\not\subseteq M_2$. Similarly, we can get an  ascending chain $M_1\subseteq M_2\subseteq ...\subseteq M_n\subseteq M_{n+1}\subseteq ...$ such that $sM_{n+1}\not\subseteq M_n$ for any $n\geq 1$. Obviously, this  ascending chain is not stationary with respective to  any $s\in S$.

  $(3)\Rightarrow (1):$ Let $N$ be a submodule of $M$ and $s\in S$ the element in $(3)$. Set $\Gamma=\{A\subseteq N| $ there is a finitely generated submodule $F_A$  of $A$ satisfies  $sA\subseteq F_A\}$. Since $0\in \Gamma$, $\Gamma$ is nonempty. Thus $\Gamma$ has a maximal element $A$. If $A\not=N$, then there is an element $x\in N-A$.  Since $F_1=F_A+Rx$ is a finitely generated submodule of $A_1=A+Rx$ such that $sA_1\subseteq F_1$, we have $F_1\in \Gamma$,  which contradicts the choice of maximality of $A$.
\end{proof}

\begin{corollary} \label{u-s-noe-ring-char}
Let $R$ be a ring and $S$ a multiplicative subset of $R$.  Then the following assertions are equivalent:
\begin{enumerate}
\item $R$ is $u$-$S$-Noetherian;
\item there exists an element $s\in S$ such that any  ascending chain of ideals of $R$ is stationary with respective to $s$;

\item  there exists an element $s\in S$ such that any nonempty of ideals of $R$ has a maximal element with respective to $s$.
\end{enumerate}
\end{corollary}

We can rediscover the following result by Proposition \ref{s-loc-u-noe-fini}.
\begin{corollary} \cite[Corollary 2.1]{as16}\label{u-s-noe-char-s}
Let $R$ be a ring and $S$ a multiplicative subset of $R$ consisting of finite elements. Then the following assertions are equivalent:
\begin{enumerate}
\item $R$ is an $S$-Noetherian ring;
\item  every increansing sequence of ideals of $R$ is $S$-stationary;
\item  every nonempty set of ideals of $R$ has an $S$-maximal element.
\end{enumerate}
\end{corollary}

Recall from \cite{z21}, an $R$-sequence  $M\xrightarrow{f} N\xrightarrow{g} L$ is called  \emph{$u$-$S$-exact} provided that there is an element $s\in S$ such that $s\Ker(g)\subseteq \Im(f)$ and $s\Im(f)\subseteq \Ker(g)$. An $R$-homomorphism $f:M\rightarrow N$ is an \emph{$u$-$S$-monomorphism}  $($resp.,   \emph{$u$-$S$-epimorphism}, \emph{$S$-isomorphism}$)$ provided $0\rightarrow M\xrightarrow{f} N$   $($resp., $M\xrightarrow{f} N\rightarrow 0$, $0\rightarrow M\xrightarrow{f} N\rightarrow 0$ $)$ is   $u$-$S$-exact. It is easy to verify that an  $R$-homomorphism $f:M\rightarrow N$ is a $u$-$S$-monomorphism $($resp., $u$-$S$-epimorphism$)$ if and only if  $\Ker(f)$ $($resp., $\Coker(f))$ is a  $u$-$S$-torsion module.

\begin{lemma}\label{s-exct-tor}\cite[Proposition 2.8]{z21}
Let $R$ be a ring and $S$ a multiplicative subset of $R$.  Let $0\rightarrow A\xrightarrow{f} B\xrightarrow{g} C\rightarrow 0$  be a $u$-$S$-exact sequence. Then  $B$ is  $u$-$S$-torsion if and only if $A$ and $C$ are $u$-$S$-torsion.
\end{lemma}

\begin{lemma}\label{s-exct-diag}
Let $R$ be a ring and $S$ a multiplicative subset of $R$.  Let
$$\xymatrix@R=20pt@C=25pt{
  0 \ar[r]^{}&A_1\ar@{^{(}->}[d]^{i_A}\ar[r]& B_1 \ar[r]^{\pi_1}\ar@{^{(}->}[d]^{i_B}&C_1\ar[r] \ar@{^{(}->}[d]^{i_C} &0\\
0 \ar[r]^{}&A_2\ar[r]&B_2 \ar[r]^{\pi_2}&C_2\ar[r] &0\\}$$
be a commutative diagram with exact rows, where $i_A, i_B$ and $i_C$ are embedding maps. Suppose $s_A A_2\subseteq A_1$ and $s_C C_2\subseteq C_1$ for some $s_A\in S, s_C\in S$. Then $s_As_CB_2\subseteq B_1$.
\end{lemma}
\begin{proof} Let $x\in B_2$. Then $\pi_2(x)\in C_2$. Thus $s_C\pi_2(x)=\pi_2(s_Cx)\in C_1$. So we have $\pi_1(y)=\pi_2(y)=\pi_2(s_Cx)$ for some $y\in B_1$. Thus $s_Cx-y=a_2$ for some $a_2\in A_2$. It follows that $s_As_Cx=s_Ay+s_Aa_2\in B_1$. Consequently,  $s_As_CB_2\subseteq B_1$.
\end{proof}

\begin{lemma} \label{s-u-noe-exact}
Let $R$ be a ring and $S$ a multiplicative subset of $R$. Let $0\rightarrow A\rightarrow B\rightarrow C\rightarrow 0$ be an exact sequence. Then $B$ is $u$-$S$-Noetherian if and only if $A$ and $C$ are $u$-$S$-Noetherian.
\end{lemma}
\begin{proof} It is easy to verify that if $B$ is $u$-$S$-Noetherian, so are $A$ and $C$. Suppose $A$ and $C$ are $u$-$S$-Noetherian. Let $\{B_i\}_{i\in \Lambda}$ be the set of all submodules of $B$. Then there exists an element $s_1\in S$ such that  $s_1(A\cap B_i)\subseteq K_i\subseteq A\cap B_i$ for some finitely generated $R$-module $K_i$ and any $i\in \Lambda$, since $A$ is $u$-$S$-Noetherian. There also exists an element $s_2\in S$ such that $s_2(B_i+A)/A\subseteq L_i\subseteq (B_i+A)/A$ for some finitely generated $R$-module $L_i$ and any $i\in \Lambda$, since $C$ is $u$-$S$-Noetherian. Let $N_i$ be the finitely generated submodule of $B_i$ generated by the finite generators of $K_i$ and finite pre-images of generators of  $L_i$.  Consider the following natural commutative diagram with exact rows: $$\xymatrix@R=20pt@C=25pt{
  0 \ar[r]^{}&K_i\ar@{^{(}->}[d]\ar[r]&N_i \ar[r]\ar@{^{(}->}[d]&L_i\ar[r] \ar@{^{(}->}[d] &0\\
0 \ar[r]^{}&A\cap B_i\ar[r]&B_i \ar[r]&(B_i+A)/A \ar[r] &0.\\}$$
Set $s=s_1s_2\in S$. We have  $sB_i\subseteq N_i\subseteq B_i$ by Lemma \ref{s-exct-diag}. So $B$ is $u$-$S$-Noetherian.
\end{proof}

\begin{proposition} \label{s-u-noe-s-exact}
Let $R$ be a ring and $S$ a multiplicative subset of $R$. Let $0\rightarrow A\rightarrow B\rightarrow C\rightarrow 0$ be a $u$-$S$-exact sequence. Then $B$ is $u$-$S$-Noetherian if and only if $A$ and $C$ are $u$-$S$-Noetherian
\end{proposition}
\begin{proof}  Let $0\rightarrow A\xrightarrow{f} B\xrightarrow{g} C\rightarrow 0$ be a $u$-$S$-exact sequence. Then there exists an element $s\in S$ such that  $ s\Ker(g)\subseteq  \Im(f)$ and $ s\Im(f)\subseteq  \Ker(g)$. Note that $\Im(f)/s\Ker(g)$ and $\Ker(g)/s\Im(f)$ are $u$-$S$-torsion. If $\Im(f)$ is $u$-$S$-Noetherian, then the submodule $s\Im(f)$ of $\Im(f)$ is $u$-$S$-Noetherian. Thus $\Ker(g)$ is $u$-$S$-Noetherian by Lemma \ref{s-u-noe-exact}. Similarly, if  $\Ker(g)$ is $u$-$S$-Noetherian, then $\Im(f)$ is $u$-$S$-Noetherian. Consider the following three exact sequences:
$0\rightarrow\Ker(g) \rightarrow  B\rightarrow \Im(g)\rightarrow 0,\quad 0\rightarrow\Im(g) \rightarrow  C\rightarrow \Coker(g)\rightarrow 0,$ and $0\rightarrow\Ker(f) \rightarrow  A\rightarrow \Im(f)\rightarrow 0$
with $\Ker(f)$ and $\Coker(g)$ $u$-$S$-torsion.   It is easy to verify that $B$ is $u$-$S$-Noetherian if and only if $A$ and $C$ are $u$-$S$-Noetherian by Lemma \ref{s-u-noe-exact}.
\end{proof}

\begin{corollary} \label{s-u-noe-u-iso}
Let $R$ be a ring, $S$ a multiplicative subset of $R$ and $ M\xrightarrow{f} N$  a $u$-$S$-isomorphism. If one of $M$ and $N$ is $u$-$S$-Noetherian, so is the other.
\end{corollary}
\begin{proof} It follows from Proposition \ref{s-u-noe-exact} since $0\rightarrow M\xrightarrow{f} N\rightarrow 0\rightarrow 0$ is a $u$-$S$-exact sequence.
\end{proof}





Let $\p$ be a prime ideal of $R$. We say an $R$-module $M$ is \emph{$u$-$\p$-Noetherian} provided that  $M$ is  $u$-$(R\setminus\p)$-Noetherian. The next result gives a local characterization of Noetherian modules.
\begin{proposition}\label{s-noe-m-loc-char}
Let $R$ be a ring and $M$ an $R$-module. Then the following statements are equivalent:
 \begin{enumerate}
\item  $M$ is Noetherian;
\item   $M$ is $u$-$\p$-Noetherian for any $\p\in \Spec(R)$;
\item   $M$ is  $u$-$\m$-Noetherian for any $\m\in \Max(R)$.
 \end{enumerate}
\end{proposition}
\begin{proof} $(1)\Rightarrow (2)\Rightarrow (3):$  Trivial.

$(3)\Rightarrow (1):$ Let $N$ be an submodule of $M$. Then for each  $\m\in \Max(R)$, there exists an element $s^{\m}\in R\setminus\m$ and a finitely generated submodule $F^{\m}$ of $N$ such that $s^{\m}N\subseteq F^{\m}$. Since $\{s^{\m} \mid \m \in \Max(R)\}$ generated $R$, there exist finite elements $\{s^{\m_1},...,s^{\m_n}\}$ such that $N=\langle s^{\m_1},...,s^{\m_n}\rangle N\subseteq F^{\m_1}+...+F^{\m_n}\subseteq N$. So $N=F^{\m_1}+...+F^{\m_n}$. It follows that $N$ is finitely generated, and thus $M$ is Noetherian.
\end{proof}

\begin{corollary}\label{s-noe-m-loc-char}
Let $R$ be a ring. Then the following statements are equivalent:
 \begin{enumerate}
\item  $R$ is a Noetherian ring;
\item   $R$ is a $u$-$\p$-Noetherian ring for any $\p\in \Spec(R)$;
\item   $R$ is  a $u$-$\m$-Noetherian ring  for any $\m\in \Max(R)$.
 \end{enumerate}
\end{corollary}

\section{$u$-$S$-Noetherian properties on some  ring constructions}

In this section, we mainly consider the $u$-$S$-Noetherian properties on trivial extensions, pullbacks and amalgamated algebras along an ideal. For more on these ring constructions, one can refer to \cite{DW09,lO14}.

Let $R$ be a commutative ring and $M$ be an $R$-module. Then the \emph{trivial extension} of $R$ by $M$ denoted by $R(+)M$ is equal to $R\bigoplus M$ as $R$-modules with coordinate-wise addition and multiplication $(r_1,m_1)(r_2,m_2)=(r_1r_2,r_1m_2+r_2m_1)$. It is easy to verify that $R(+)M$ is a commutative ring with identity $(1,0)$. Let $S$ be a multiplicative subset of $R$. Then it is easy to verify that $S(+)M=\{(s,m)|s\in S, m\in M\}$ is a multiplicative subset of $R(+)M$. Now, we give a $u$-$S$-Noetherian property on the trivial extension.

\begin{proposition}\label{trivial extension-usn} Let $R$ be a commutative ring, $S$ a multiplicative subset of $R$ and $M$  an $R$-module. Then $R(+)M$ is a $u$-$S(+)M$-Noetherian ring if and only if $R$ is a $u$- $S$-Noetherian ring and $M$ is a $u$-$S$-Noetherian $R$-module.
\end{proposition}
\begin{proof} Note that we have an exact sequence of  $R(+)M$-modules: $$0\rightarrow 0(+)M\xrightarrow{i} R(+)M\xrightarrow{\pi} R\rightarrow 0.$$ Suppose $R(+)M$ is a $u$-$S(+)M$-Noetherian ring. Let $\{I_i\}_{i\in \Lambda}$ be the set of all ideals of $R$. Then $\{I_i(+)M\}_{i\in \Lambda}$ is a set of ideals of $R(+)M$. So there is an element $(s,m)\in S(+)M$ and  finitely generated sub-ideals $O_i$  of $I_i(+)M$ such that $(s,m)I_i(+)M\subseteq O_i$. Thus  $sI_i\subseteq \pi(O_i)\subseteq I_i$. Suppose $O_i$ is  generated by $\{(r_{1,i},m_{1,i}),...,(r_{n,i},m_{n,i})\}$. Then it is easy to verify that $\pi(O_i)$  is generated by $\{r_{1,i},...,r_{n,i}\}$. So $R$ is a $u$-$S$-Noetherian ring. Let $\{M_i\}_{i\in \Gamma}$ be the set of all submodules of $M$. Then $\{0(+)M_i\}_{i\in \Gamma}$ is a set of  ideals of $R(+)M$. Thus there is an element $(s',m')\in S(+)M$ and  finitely generated sub-ideals $O'_i$  of $0(+)M_i$ such that $(s',m')0(+)M_i\subseteq O'_i$. So $s'M_i\subseteq N_i\subseteq M_i$ where  $0(+)N_i=O'_i$. Suppose that $O'_i$ is generated by $\{(r'_{1,i},m'_{1,i}),...,(r'_{n,i},m'_{n,i}\}$. Then it is easy to verify that $N_i$ is generated by $\{m'_{1,i},...,m'_{n,i}\}$. Thus $M$ is a $u$-$S$-Noetherian $R$-module.

Suppose $R$ is a $u$-$S$-Noetherian ring and $M$ is a $u$-$S$-Noetherian $R$-module. Let $O^{\bullet}: O_1\subseteq O_2\subseteq ...$ be  an ascending chain of ideals of $R(+)M$. Then there is  an ascending chain of ideals of $R$: $\pi(O^{\bullet}): \pi(O_1)\subseteq \pi(O_2)\subseteq ...$. Thus there is an element $s\in S$  which is independent of $O^{\bullet}$ satisfying that there exists $k\in \mathbb{Z}^{+}$ such that  $s\pi(O_n)\subseteq \pi(O_k)$ for any $n\geq k$. Similarly, $O^{\bullet}\cap 0(+)M:   O_1\cap 0(+)M\subseteq O_2\cap 0(+)M\subseteq ...$ is  an ascending chain of sub-ideals of $0(+)M$ which are equivalent to some submodules of $M$. So there is an element $s'\in S$  satisfying that there exists $k'\in \mathbb{Z}^{+}$ such that $s'O_n\cap 0(+)M\subseteq O_k\cap 0(+)M$ for any $n\geq k'$. Let $l=\max(k,k')$ and $n\geq l$. Consider the following natural commutative diagram with exact rows:
$$\xymatrix@R=20pt@C=25pt{
  0 \ar[r]^{}&O_l\cap 0(+)M \ar@{^{(}->}[d]\ar[r]&O_l \ar[r]\ar@{^{(}->}[d]&\pi(O_l)\ar[r] \ar@{^{(}->}[d] &0\\
0 \ar[r]^{}&O_n\cap 0(+)M \ar[r]&O_n \ar[r]&\pi(O_n) \ar[r] &0.\\}$$
Set $t=ss'$. Then we have $tO_n\subseteq O_l$ for any $n\geq l$ by Lemma \ref{s-exct-diag}. So $R(+)M$ is a $u$- $S(+)M$-Noetherian ring  by Theorem \ref{u-s-noe-char}.
\end{proof}

Let $\alpha: A\rightarrow C$ and $\beta: B\rightarrow C$  be ring homomorphisms.  Then the subring $$D:= \alpha \times_C \beta:= \{(a, b)\in  A\times B | \alpha(a) =\beta(b)\}$$ of $A\times B$ is called the \emph{pullback}  of $\alpha$ and $\beta$. Let $D$ be a pullback of $\alpha$ and $\beta$. Then there is a pullback diagram in the category of commutative rings:
$$\xymatrix@R=20pt@C=25pt{
D\ar[d]^{p_B}\ar[r]^{p_A}& A \ar[d]^{\alpha}\\
B\ar[r]^{\beta}&C. \\
}$$
If $S$ is a multiplicative subset of $D$, then it is easy to verify that $p_A(S):=\{p_A(s)\in A|s\in S\}$ is a  multiplicative subset of $A$. Now, we give a $u$-$S$-Noetherian property on the pullback diagram.

\begin{proposition}\label{pullback-usn} Let $\alpha: A\rightarrow C$ be a ring homomorphism and $\beta: B\rightarrow C$  a  surjective ring homomorphism. Let $D$ be the pullback of $\alpha$ and $\beta$. If $S$ is a multiplicative subset of $D$, then
the following assertions are equivalent:
 \begin{enumerate}
\item  $D$ is a $u$-$S$-Noetherian ring;
\item  $A$ is a $u$-$p_A(S)$-Noetherian ring and $\Ker(\beta)$ is a $u$-$S$-Noetherian $D$-module.
 \end{enumerate}
\end{proposition}
\begin{proof} Let $D$ be the pullback of $\alpha$ and $\beta$. Since $\beta$ is a  surjective ring homomorphism, so is $p_A$.  Then there is a short exact sequence  of $D$-modules:
$$0\rightarrow \Ker(\beta)\rightarrow D\rightarrow A\rightarrow 0.$$ By Proposition \ref{s-u-noe-s-exact}, $D$ is a $u$-$S$-Noetherian $D$-module if and only if $\Ker(\beta)$ and $A$ are $u$-$S$-Noetherian $D$-modules. Since $p_A$ is surjective, the $D$-submodules of $A$ are exactly the ideals of the ring $A$. Thus $A$ is a $u$-$S$-Noetherian $D$-module if and only if $A$ is a $u$-$p_A(S)$-Noetherian ring.
\end{proof}

Let $f:A\rightarrow B$ be a ring homomorphism and $J$ an ideal of $B$. Following from \cite{df09} the  \emph{amalgamation} of $A$ with $B$ along $J$ with respect to $f$, denoted by $A\bowtie^fJ$, is defined as $$A\bowtie^fJ=\{(a,f(a)+j)|a\in A,j\in J\},$$  which is  a subring of of $A \times B$.  Following from \cite[Proposition 4.2]{df09}, $A\bowtie^fJ$  is the pullback $\widehat{f}\times_{B/J}\pi$,
 where $\pi:B\rightarrow B/J$ is the natural epimorphism and $\widehat{f}=\pi\circ f$:
$$\xymatrix@R=20pt@C=25pt{
A\bowtie^fJ\ar[d]^{p_B}\ar[r]_{p_A}& A\ar[d]^{\widehat{f}}\\
B\ar[r]^{\pi}&B/J. \\
}$$
For a multiplicative subset $S$ of $A$, set $S^\prime:= \{(s, f (s)) | s \in S\},$ and $f(S):=\{f(s)\in B|s\in S\}$. Then it is easy to verify that $S^\prime$ and $f(S)$ are multiplicative subsets of $A\bowtie^fJ$  and $B$ respectively.

\begin{lemma} \label{s-u-noe-epi}
Let $\alpha:R\rightarrow R^\prime$ be a surjective ring homomorphism and $S$ a multiplicative subset of $R$. If $R$ is a  $u$-$S$-Noetherian ring, then $R^\prime$ is a $u$-$\alpha(S)$-Noetherian ring.
\end{lemma}
\begin{proof} Since $R$ is  $u$-$S$-Noetherian, there is an element $s\in S$ such that for any ideal $J$ of $R$, there exists a finitely generated sub-ideal $F_J$ of $J$ satisfying $sJ\subseteq F_J$. Let $I$ be an ideal of $R^\prime$. Since $\alpha:R\rightarrow R^\prime$ is a surjective ring homomorphism, there exists an ideal $\alpha^{-1}(I)$ of $R$ such that $\alpha(\alpha^{-1}(I))=I$. Thus there exists a finitely generated sub-ideal $F_{\alpha^{-1}(I)}$ of $\alpha^{-1}(I)$ satisfying $s\alpha^{-1}(I)\subseteq F_{\alpha^{-1}(I)}$.  So $\alpha(F_{\alpha^{-1}(I)})$ is a finitely generated sub-ideal  of $I$ satisfying $\alpha(s)I\subseteq \alpha(F_{\alpha^{-1}(I)})$.
\end{proof}

\begin{proposition}\label{amag-usn}  Let $f :A\rightarrow B$ be a ring homomorphism, $J$ an ideal of $B$ and $S$ a multiplicative subset of $A$. Set $S^\prime= \{(s, f (s)) | s \in S\}$ and $f(S)=\{f(s)\in B|s\in S\}$. Then the following statements
are equivalent:
 \begin{enumerate}
\item  $A\bowtie^fJ$ is a $u$-$S^\prime$-Noetherian ring;
\item  $A$ is a $u$-$S$-Noetherian ring and $J$ is a $u$-$S^\prime$-Noetherian $A\bowtie^fJ$-module $($with the $A\bowtie^fJ$-module structure naturally induced by $p_B$, where $p_B : A\bowtie^fJ\rightarrow B$ defined by $(a,f(a)+j)\rightarrow f(a)+j)$;
\item  $A$ is a $u$-$S$-Noetherian ring and $f(A)+J$ is a $u$-$f(S)$-Noetherian ring.
\end{enumerate}
\end{proposition}
\begin{proof}  $(1)\Leftrightarrow(2)$ Follows from  Proposition \ref{pullback-usn}.

 $(1)\Rightarrow(3)$: By Proposition \ref{pullback-usn}, $A$ is a $u$-$S$-Noetherian ring. By \cite[Proposition 5.1]{df09}, there is a short exact sequence $0\rightarrow f^{-1}(J)\times \{0\}\rightarrow A\bowtie^fJ \rightarrow f(A)+J\rightarrow 0$ of $A\bowtie^fJ$-modules.  Note that any $A\bowtie^fJ$-submodule of  $f(A)+J$ is exactly an ideal of $f(A)+J$. Since $p_B(S^\prime)=f(S),$  we conclude that  $f(A)+J$ is a $u$- $f(S)$-Noetherian ring by Proposition \ref{s-u-noe-s-exact}.

$(3)\Rightarrow(2)$:  Let $f(s)$ be an element in $f(S)$ such that for any ideal of $f(A)+J$   is $u$- $f(S)$-Noetherian with respective to $f(s)$. Then for any  $A\bowtie^fJ$-submodule $J_0$ of $J$, $J_0$ is an ideal of $f(A)+J$ since every $A\bowtie^fJ$-submodule of $J$ is an ideal of $f(A)+J$. Since $f(A)+J$ is $f(S)$-Noetherian, there exists $j_1,...,j_k\in J_0$ such that $f(s)J_0\subseteq \langle j_1,...,j_k\rangle (f(A)+J)\subseteq J_0$. Hence we obtain
  $$(s,f(s))J_0\subseteq A\bowtie^fJ j_1+...+A\bowtie^fJ j_k\subseteq J_0.$$
Thus $J$ is $u$-$S^\prime$-Noetherian with respective to  $(s,f(s)).$
\end{proof}

\section{Cartan-Eilenberg-Bass Theorem for uniformly $S$-Noetherian rings}
It is well known that  an $R$-module $E$ is \emph{injective} provided that the induced sequence $0\rightarrow \Hom_R(C,E)\rightarrow \Hom_R(B,E)\rightarrow \Hom_R(A,E)\rightarrow 0$ is exact for any exact sequence $0\rightarrow A\rightarrow B\rightarrow C\rightarrow 0$. The well-known Cartan-Eilenberg-Bass Theorem says that a ring $R$ is Noetherian if and only if any direct sum of  injective modules  is injective (see \cite[Theorem 3.1.17]{ej11}). In order to obtain the Cartan-Eilenberg-Bass Theorem for uniformly $S$-Noetherian rings, we first introduce the $S$-analogue of  injective modules.

\begin{definition} Let $R$ be a ring and $S$ a multiplicative subset of $R$. An $R$-module $E$ is called $u$-$S$-injective $($abbreviates uniformly $S$-injective$)$ provided that the induced sequence $$0\rightarrow \Hom_R(C,E)\rightarrow \Hom_R(B,E)\rightarrow \Hom_R(A,E)\rightarrow 0$$ is $u$-$S$-exact for any $u$-$S$-exact sequence $0\rightarrow A\rightarrow B\rightarrow C\rightarrow 0$.
\end{definition}

\begin{lemma}\label{u-S-tor-ext} Let  $R$ be a ring and $S$ a multiplicative subset of $R$. If $T$ is a $u$-$S$-torsion module, then $\Ext_R^{n}(T,M)$ and $\Ext_R^{n}(M,T)$ are $u$-$S$-torsion for any $R$-module $M$ and any $n\geq 0$.
\end{lemma}
\begin{proof} We only prove $\Ext_R^{n}(T,M)$ is  $u$-$S$-torsion, Since the case of $\Ext_R^{n}(M,T)$ is similar. Let $T$ be a $u$-$S$-torsion module with $sT=0$. If $n=0$, then for any $f\in \Hom_R(T,M)$, we have $sf(t)=f(st)=0$ for any $t\in T$. Thus $sf=0$ and so $s\Hom_R(T,M)=0$. Let $0\rightarrow M\rightarrow E\rightarrow \Omega^{-1} (M)\rightarrow0$ be a short exact sequence with $E$ injective and $\Omega^{-1} (M)$ the $1$-st cosyzygy of $M$. Then $\Ext_R^{1}(T,M)$ is a quotient of $\Hom_R(T,\Omega^{-1} (M))$ which is $u$-$S$-torsion. Thus $\Ext_R^{1}(T,M)$ is $u$-$S$-torsion. For $n\geq 2$, we have an isomorphism $\Ext_R^{n}(T,M)\cong \Ext_R^{1}(T,\Omega^{-(n-1)}(M))$ where $\Omega^{-(n-1)}(M)$ is the $(n-1)$-th cosyzygy of $M$. Since $\Ext_R^{1}(T,\Omega^{-(n-1)}(M))$ is  $u$-$S$-torsion by induction, $\Ext_R^{n}(T,M)$ is  $u$-$S$-torsion.
\end{proof}

\begin{theorem}\label{s-inj-ext}
Let $R$ be a ring, $S$ a multiplicative subset of $R$ and $E$ an $R$-module. Then the following assertions are equivalent:
\begin{enumerate}
\item  $E$ is  $u$-$S$-injective;

\item for any short exact sequence $0\rightarrow A\xrightarrow{f} B\xrightarrow{g} C\rightarrow 0$, the induced sequence $0\rightarrow \Hom_R(C,E)\xrightarrow{g^\ast} \Hom_R(B,E)\xrightarrow{f^\ast} \Hom_R(A,E)\rightarrow 0$ is  $u$-$S$-exact;

\item  $\Ext_R^1(M,E)$ is  $u$-$S$-torsion for any  $R$-module $M$;

\item  $\Ext_R^n(M,E)$ is  $u$-$S$-torsion for any  $R$-module $M$ and $n\geq 1$.

\end{enumerate}
\end{theorem}
\begin{proof} $(1)\Rightarrow(2)$ and $(4)\Rightarrow(3)$: Trivial.

$(2)\Rightarrow(3)$:  Let $0\rightarrow L\rightarrow P\rightarrow M\rightarrow 0$ be a short exact sequence with $P$ projective. Then there exists a long exact sequence  $0\rightarrow \Hom_R(M,E)\rightarrow \Hom_R(P,E)\rightarrow \Hom_R(L,E)\rightarrow \Ext_R^1(M,E) \rightarrow 0$.  Thus  $\Ext_R^1(M,E)$ is  $u$-$S$-torsion by $(2)$.

$(3)\Rightarrow (2)$: Let  $0\rightarrow A\xrightarrow{f} B\xrightarrow{g} C\rightarrow 0$ be a short exact sequence. Then we have a long exact sequence $0\rightarrow \Hom_R(C,E)\xrightarrow{g^\ast} \Hom_R(B,E)\xrightarrow{f^\ast} \Hom_R(A,E)\xrightarrow{\delta} \Ext_R^1(C,E) \rightarrow 0$.   By  $(3)$, $\Ext_R^1(C,E)$ is  $u$-$S$-torsion,  and so $0\rightarrow \Hom_R(C,E)\xrightarrow{g^\ast} \Hom_R(B,E)\xrightarrow{f^\ast} \Hom_R(A,E)\rightarrow 0$ is  $u$-$S$-exact.

$(3)\Rightarrow(4)$:  Let $M$ be an $R$-module. Denote by $\Omega^{n-1}(M)$ the $(n-1)$-th syzygy of $M$. Then $\Ext_R^n(M,E)\cong \Ext_R^1(\Omega^{n-1}(M),E)$ is $u$-$S$-torsion by $(3)$.

$(2)\Rightarrow(1)$: Let $E$ be an $R$-module satisfies $(2)$.  Suppose $0\rightarrow A\xrightarrow{f} B\xrightarrow{g} C\rightarrow 0$ is a $u$-$S$-exact sequence. Then there is an exact sequence  $B\xrightarrow{g} C\rightarrow T\rightarrow 0 $ where $T=\Coker(g)$ is $u$-$S$-torsion. Then we have an exact sequence $$0\rightarrow \Hom_R(T,E)\rightarrow \Hom_R(C,E)\rightarrow \Hom_R(B,E).$$
By Lemma \ref{u-S-tor-ext}, we have $ \Hom_R(T,E)$ is $u$-$S$-torsion. So $0\rightarrow \Hom_R(C,E)\xrightarrow{g^\ast} \Hom_R(B,E)\xrightarrow{f^\ast} \Hom_R(A,E)\rightarrow 0$  is  $u$-$S$-exact at $\Hom_R(C,E)$.

There are also two short exact sequences:
\begin{center}
$0\rightarrow \Ker(f)\xrightarrow{i_{A}} A\xrightarrow{\pi_{\Im(f)}} \Im(f)\rightarrow 0$ and $0\rightarrow \Im(f)\xrightarrow{i_B} B\rightarrow \Coker(f)\rightarrow 0,$
\end{center}
where $\Ker(f)$ is $u$-$S$-torsion. Consider the  induced exact sequences  $$0\rightarrow  \Hom_R(\Im(f),E)\xrightarrow{\pi_{\Im(f)}^\ast}  \Hom_R(A,E)\xrightarrow{i_{A}^\ast} \Hom_R(\Ker(f),E)$$ and
$$0\rightarrow \Hom_R(\Coker(f),E)\rightarrow \Hom_R(B,E)\xrightarrow{i_{B}^\ast}  \Hom_R(\Im(f),E).$$ Then  $\Im(i_{A}^\ast)$ and $\Coker(i_{B}^\ast)$ are all $u$-$S$-torsion.
We have the following pushout diagram:

$$\xymatrix@R=20pt@C=25pt{ & 0\ar[d]&0\ar[d]&&\\
 & \Im(i_{B}^\ast)\ar[d]\ar@{=}[r]^{} &\Im(i_{B}^\ast)\ar[d]& &  \\
  0 \ar[r]^{}& \Hom_R(\Im(f),E)\ar[d]\ar[r]& \Hom_R(A,E) \ar[r]\ar[d]&\Im(i_{A}^\ast)\ar[r] \ar@{=}[d] &0\\
0 \ar[r]^{}&\Coker(i_{B}^\ast)\ar[d]\ar[r]&Y \ar[r]\ar[d]&\Im(i_{A}^\ast)\ar[r] &0\\
 & 0 &0 & &\\}$$
Since $\Im(i_{A}^\ast)$ and $\Coker(i_{B}^\ast)$ are all $u$-$S$-torsion, $Y$ is also $u$-$S$-torsion by Lemma \ref{s-exct-tor}. Thus the natural composition $f^\ast: \Hom_R(B,E)\rightarrow \Im(i_{B}^\ast)\rightarrow  \Hom_R(A,E)$ is a $u$-$S$-epimorphism. So $0\rightarrow \Hom_R(C,E)\xrightarrow{g^\ast} \Hom_R(B,E)\xrightarrow{f^\ast} \Hom_R(A,E)\rightarrow 0$ is  $u$-$S$-exact at $\Hom_R(A,E)$.

Since the sequence $0\rightarrow A\xrightarrow{f} B\xrightarrow{g} C\rightarrow 0$ is $u$-$S$-exact at $B$ and $C$, there exists  $s\in S$ such that   $s\Ker(g)\subseteq \Im(f)$, $s\Im(f)\subseteq \Ker(g)$ and $s\Coker(g)=0$. We claim that $s^2\Im (g^\ast)\subseteq \Ker(f^\ast)$ and $s^2\Ker(f^\ast)\subseteq \Im (g^\ast)$. Indeed, consider the following diagram:
$$\xymatrix@R=20pt@C=25pt{
  & & E& &\\
0 \ar[r]^{}&A\ar[r]^{f}&B  \ar[u]^{h}\ar[r]^{g}&C\ar[r] &0\\}$$
Suppose $h\in \Im (g^\ast)$. Then there exists $u\in \Hom_R(C,E)$ such that $h=u\circ g$. Thus for any $a\in A$, $sh\circ f (a)=su\circ g\circ f(a)=u\circ g\circ sf(a)=0$ since $s\Im(f)\subseteq \Ker(g)$. So $sh\circ f=0$ and then  $s\Im (g^\ast)\subseteq \Ker(f^\ast)$. Thus $s^2\Im (g^\ast)\subseteq \Ker(f^\ast)$.   Now, suppose $h\in \Ker(f^\ast)$. Then $h\circ f=0$. Thus $\Ker(h)\supseteq \Im(f)\supseteq s\Ker(g)$. So $sh\circ i_{\Ker(g)}=0$ where $i_{\Ker(g)}:\Ker(g)\hookrightarrow B$ is the natural embedding map. There is a well-defined $R$-homomorphism $v:\Im(g)\rightarrow E$ such that $v\circ \pi_B=sh$, where $\pi_B$ is the natural epimorphism $B\twoheadrightarrow \Im(g)$. Consider the exact sequence $\Hom_R(\Coker(g),E)\rightarrow \Hom_R(C,E)\rightarrow \Hom_R(\Im(g),E)\rightarrow \Ext_R^1(\Coker(g),E)$ induced by $0\rightarrow \Im(g)\rightarrow C\rightarrow \Coker(g)\rightarrow 0$. Since $s\Hom_R(\Coker(g),E)=s \Ext_R^1(\Coker(g),E)=0$, $s\Hom_R(\Im(g),E)\subseteq i_{\Im(g)}^\ast(\Hom_R(C,E))$.  Thus there is a homomorphism $u:C\rightarrow E$ such that $s^2h=v\circ g$. Then we have $s^2\Ker(f^\ast)\subseteq \Im (g^\ast)$. So $0\rightarrow \Hom_R(C,E)\xrightarrow{g^\ast} \Hom_R(B,E)\xrightarrow{f^\ast} \Hom_R(A,E)\rightarrow 0$ is  $u$-$S$-exact at $\Hom_R(B,E)$.
\end{proof}

It follows from Theorem \ref{s-inj-ext} that  $u$-$S$-torsion modules and injective modules are $u$-$S$-injective.
\begin{corollary}\label{inj-ust-s-inj}
Let $R$ be a ring and $S$ a multiplicative subset of $R$. Suppose $E$ is a $u$-$S$-torsion $R$-module or an injective $R$-module. Then $E$ is $u$-$S$-injective.
\end{corollary}

The following example shows that the condition ``$\Ext_R^1(M,F)$ is  $u$-$S$-torsion for any  $R$-module $M$'' in Theorem \ref{s-inj-ext} can not be replaced by ``$\Ext_R^1(R/I,F)$ is  $u$-$S$-torsion for any  ideal $I$ of  $R$''.
\begin{example}\label{uf not-extsion}
Let $R=\mathbb{Z}$ be the ring of integers, $p$ a prime in $\mathbb{Z}$ and  $S=\{p^n|n\geq 0\}$. Let $J_p$ be the additive group of all $p$-adic integers $($see \cite{FS15} for example$)$. Then $\Ext_{R}^1(R/I,J_p)$ is $u$-$S$-torsion for any  ideal $I$ of  $R$. However, $J_p$ is not  $u$-$S$-injective.
\end{example}
\begin{proof}
Let $\langle n\rangle$ be an ideal of $\mathbb{Z}$. Suppose $n=p^km$ with $(p,m)=1$. Then $\Ext_{\mathbb{Z}}^1(\mathbb{Z}/\langle n\rangle, J_p)\cong J_p/nJ_p\cong \mathbb{Z}/\langle p^k\rangle$ by \cite[Exercise 1.3(10)]{FS15}. So $\Ext_{\mathbb{Z}}^1(\mathbb{Z}/\langle n\rangle, J_p)$ is $u$-$S$-torsion for any  ideal $\langle n\rangle$ of  $\mathbb{Z}$. However,  $J_p$ is not  $u$-$S$-injective. Indeed, let $\mathbb{Z}(p^{\infty})$ be the quasi-cyclic group (see \cite{FS15} for example). Then $\mathbb{Z}(p^{\infty})$ is a divisible group and $J_p\cong\Hom_{\mathbb{Z}}(\mathbb{Z}(p^{\infty}),\mathbb{Z}(p^{\infty})))$. So
\begin{align*}
  &\Ext_{\mathbb{Z}}^1(\mathbb{Z}(p^{\infty}), M) \\
 \cong &\Ext_{\mathbb{Z}}^1(\mathbb{Z}(p^{\infty}), \Hom_{\mathbb{Z}}(\mathbb{Z}(p^{\infty}),\mathbb{Z}(p^{\infty}))) \\
   \cong&\Hom_{\mathbb{Z}}(\Tor_1^{\mathbb{Z}}(\mathbb{Z}(p^{\infty}),\mathbb{Z}(p^{\infty})),\mathbb{Z}(p^{\infty}))\\
   \cong&\Hom_{\mathbb{Z}}(\mathbb{Z}(p^{\infty}),\mathbb{Z}(p^{\infty}))\cong J_p.
\end{align*}
Note that for any $p^k\in S$, we have $p^kJ_p\not=0$. So $J_p$ is not $u$-$S$-injective.
\end{proof}

\begin{remark}\label{uf not-dprod} It is well known that any direct  product of injective modules is injective. However, the direct  product of $u$-$S$-injective modules need not be $u$-$S$-injective. Indeed, Let $R$ and $S$ be in Example \ref{uf not-extsion}. Let $\mathbb{Z}/\langle p^k\rangle$ be cyclic group of order $p^k$ ($k\geq 1$). Then each $\mathbb{Z}/\langle p^k\rangle$ is  $u$-$S$-torsion, and thus is $u$-$S$-injective. Let $\mathbb{Q}$ be the rational number group. Then, by \cite[Chapter 9 Theorem 6.2]{FS15}, we have
$\Ext_{\mathbb{Z}}(\mathbb{Q}/\mathbb{Z},\prod\limits_{k=1}^\infty \mathbb{Z}/\langle p^k\rangle)\cong \prod\limits_{k=1}^\infty \Ext_{\mathbb{Z}}^1(\mathbb{Q}/\mathbb{Z}, \mathbb{Z}/\langle p^k\rangle)\cong  \prod\limits_{k=1}^\infty \mathbb{Z}/\langle p^k\rangle$ since each $\mathbb{Z}/\langle p^k\rangle$ is a reduced cotorsion group. It is easy to verify that $\prod\limits_{k=1}^\infty \mathbb{Z}/\langle p^k\rangle$ is not  $u$-$S$-torsion. So $\prod\limits_{k=1}^\infty \mathbb{Z}/\langle p^k\rangle$ is not $u$-$S$-injective.
\end{remark}

\begin{proposition}\label{s-inj-prop}
Let $R$ be a ring and $S$ a multiplicative subset of $R$. Then the following assertions hold.
\begin{enumerate}
\item Any finite direct sum of  $u$-$S$-injective modules is  $u$-$S$-injective.
\item Let $0\rightarrow A\xrightarrow{f} B\xrightarrow{g} C\rightarrow 0$  be a $u$-$S$-exact sequence. If $A$ and $C$ are  $u$-$S$-injective modules, so is $B$.
\item  Let $A\rightarrow B$ be a $u$-$S$-isomorphism. If one of $A$ and $B$ is  $u$-$S$-injective, so is the other.
\item Let $0\rightarrow A\xrightarrow{f} B\xrightarrow{g} C\rightarrow 0$  be a $u$-$S$-exact sequence.  If $A$ and $B$ are  $u$-$S$-injective, then $C$ is  $u$-$S$-injective.
\end{enumerate}
\end{proposition}
\begin{proof}
$(1)$ Suppose $E_1,...,E_n$ are  $u$-$S$-injective modules. Let $M$ be an $R$-module. Then there exists  $s_i\in S$ such that  $s_i\Ext_R^1(M,E_i)=0$ for each $i=1,...,n$. Set $s=s_1...s_n$. Then $s\Ext_R^1(M,\bigoplus\limits_{i=1}^n E_i)\cong \bigoplus\limits_{i=1}^ns\Ext_R^1(M, E_i)=0$. Thus $\bigoplus\limits_{i=1}^n E_i$ is  $u$-$S$-injective.

$(2)$ Suppose  $A$ and $C$ are  $u$-$S$-injective modules and   $0\rightarrow A\xrightarrow{f} B\xrightarrow{g} C\rightarrow 0$  is a $u$-$S$-exact sequence. Then there are three short exact sequences: $0\rightarrow \Ker(f)\rightarrow A\rightarrow \Im(f)\rightarrow 0$, $0\rightarrow \Ker(g)\rightarrow B\rightarrow \Im(g)\rightarrow 0$ and $0\rightarrow \Im(g)\rightarrow C\rightarrow \Coker(g)\rightarrow 0$. Then  $\Ker(f)$ and $\Coker(g)$ are all $u$-$S$-torsion  and $s\Ker(g)\subseteq \Im(f)$ and $s\Im(f)\subseteq \Ker(g)$ for some $s\in S$.  Let $M$ be an $R$-module.  Then $$ \Ext_R^1(M,A)\rightarrow \Ext_R^1(M,\Im(f))\rightarrow \Ext_R^2(M,\Ker(f))$$ is exact. Since $\Ker(f)$ is $u$-$S$-torsion and $A$ is  $u$-$S$-injective,  $\Ext_R^1(M,\Im(f))$ is $u$-$S$-torsion.
Note $$\Hom_R(M,\Coker(g))\rightarrow \Ext_R^1(M,\Im(g))\rightarrow \Ext_R^1(M,C)$$  is exact. Since $\Coker(g)$ is   $u$-$S$-torsion,  $\Hom_R(M,\Coker(g))$ is  $u$-$S$-torsion by Lemma \ref{u-S-tor-ext}. Thus $\Ext_R^1(M,\Im(g))$ is  $u$-$S$-torsion as $\Ext_R^1(M,C)$ is $u$-$S$-torsion. We also note that
$$\Ext_R^1(M,\Ker(g))\rightarrow \Ext_R^1(M,B) \rightarrow \Ext_R^1(M,\Im(g))$$ is exact. Thus to verify that $\Ext_R^1(M,B)$ is $u$-$S$-torsion, we just need to show $\Ext_R^1(M,\Ker(g))$ is $u$-$S$-torsion. Denote $N= \Ker(g)+\Im(f)$.
Consider the following two exact sequences
\begin{center}
$0\rightarrow \Ker(g)\rightarrow N\rightarrow N/\Ker(g)\rightarrow 0$ and $0\rightarrow \Im(f)\rightarrow N\rightarrow  N/\Im(f)\rightarrow 0.$
\end{center}
Then it is easy to verify $N/\Ker(g)$ and $N/\Im(f)$ are all  $u$-$S$-torsion. Consider the following induced two exact sequences
$$\Hom_R(M,N/\Im(f))\rightarrow \Ext_R^1(M,\Ker(g)) \rightarrow \Ext_R^1(M,N)  \rightarrow \Ext_R^1(M, N/\Im(f)),$$ $$\Hom_R(M,N/\Ker(g)) \rightarrow \Ext_R^1(M,\Im(f)) \rightarrow \Ext_R^1(M,N) \rightarrow \Ext_R^1(M, N/\Ker(g)).$$ Thus $\Ext_R^1(M,\Ker(g))$ is $u$-$S$-torsion if and only if  $\Ext_R^1(M,\Im(f))$ is $u$-$S$-torsion. Consequently,  $B$ is  $u$-$S$-injective since $\Ext_R^1(M,\Im(f))$ is $u$-$S$-torsion.

$(3)$ Considering the $u$-$S$-exact sequences $0\rightarrow A\rightarrow B\rightarrow 0\rightarrow 0$ and  $0 \rightarrow 0\rightarrow  A\rightarrow B\rightarrow 0$, we have $A$ is $u$-$S$-injective if and only if  $B$ is  $u$-$S$-injective by $(2)$.

(4) Suppose $0\rightarrow A\xrightarrow{f} B\xrightarrow{g} C\rightarrow 0$  is a $u$-$S$-exact sequence. Then, as in the proof of $(3)$, there are three short exact sequences: $0\rightarrow \Ker(f)\rightarrow A\rightarrow \Im(f)\rightarrow 0$, $0\rightarrow \Ker(g)\rightarrow B\rightarrow \Im(g)\rightarrow 0$ and $0\rightarrow \Im(g)\rightarrow C\rightarrow \Coker(g)\rightarrow 0$. Then  $\Ker(f)$ and $\Coker(g)$ are all $u$-$S$-torsion  and $s\Ker(g)\subseteq \Im(f)$ and $s\Im(f)\subseteq \Ker(g)$ for some $s\in S$.
Let $M$ be an $R$-module. Note that $$\Hom_R(M,\Coker(g))\rightarrow\Ext_R^1(M,\Im(g))\rightarrow\Ext_R^1(M,C)\rightarrow \Ext_R^1(M,\Coker(g)) $$ is exact. Since $\Coker(g)$ is $u$-$S$-torsion, we have $\Hom_R(M,\Coker(g)$  and  $\Ext_R^1(M,\Coker(g))$  are  $u$-$S$-torsion by  Lemma \ref{u-S-tor-ext}. We just need to verify $\Ext_R^1(M,\Im(g))$ is  $u$-$S$-torsion.  Note that $$\Ext_R^1(M,B)\rightarrow \Ext_R^1(M,\Im(g)) \rightarrow \Ext_R^2(M,\Ker(g))$$ is exact. Since $\Ext_R^1(M,B)$ is $u$-$S$-torsion, we just need to verify that $\Ext_R^2(M,\Ker(g))$ is $u$-$S$-torsion. By the proof of $(2)$, we just need to show that $\Ext_R^2(M,\Im(f))$ is  $u$-$S$-torsion.  Note that $$ \Ext_R^2(M,A)\rightarrow \Ext_R^2(M,\Im(f))\rightarrow \Ext_R^3(M,\Ker(f))$$ is exact. Since $\Ext_R^2(M,A)$ and $\Ext_R^3(M,\Ker(f))$  is $u$-$S$-torsion, we have  $\Ext_R^2(M,\Im(f))$ is $u$-$S$-torsion. So $C$ is $u$-$S$-injective.
\end{proof}

Let $\p$ be a prime ideal of $R$. We say an $R$-module $E$ is \emph{$u$-$\p$-injective} shortly provided that  $E$ is  $u$-$(R\setminus\p)$-injective. The next result gives a local characterization of injective modules.
\begin{proposition}\label{s-injective-loc-char}
Let $R$ be a ring and $E$ an $R$-module. Then the following statements are equivalent:
 \begin{enumerate}
\item  $E$ is injective;
\item   $E$ is    $u$-$\p$-injective for any $\p\in \Spec(R)$;
\item   $E$ is   $u$-$\m$-injective for any $\m\in \Max(R)$.
 \end{enumerate}
\end{proposition}
\begin{proof} $(1)\Rightarrow (2)$ It follows from Theorem \ref{s-inj-prop}.

$(2)\Rightarrow (3):$  Trivial.

 $(3)\Rightarrow (1):$ Let $M$ be an $R$-module. Then $\Ext_R^1(M,E)$ is  $u$-$(R-\m)$-torsion. Thus for any $\m\in \Max(R)$, there exists  $s_{\m}\in S$ such that $s_{\m}\Ext_R^1(M,E)=0$. Since the ideal generated by all $s_{\m}$ is $R$, $\Ext_R^1(M,E)=0$. So $E$ is injective.
\end{proof}

We say an $R$-module $M$ is $S$-divisible is $M=sM$ for any $s\in S$. The well known Baer's Criterion states that an $R$-module $E$ is injective if and only if $\Ext_R^1(R/I,E)=0$ for any ideal $I$ of $R$. The next result gives a  uniformly $S$-version of Baer's Criterion.
\begin{proposition}\label{s-inj-baer}{\bf (Baer's Criterion for $u$-$S$-injective modules)}
Let $R$ be a ring, $S$ a multiplicative subset of $R$ and $E$ an  $R$-module. If $E$ is a $u$-$S$-injective module then there exists an element $s\in S$ such that $s\Ext_R^1(R/I,E)=0$   for any ideal $I$ of $R$. Moreover, if $E=sE$, then the converse also holds.
\end{proposition}
\begin{proof}

If $E$ is a $u$-$S$-injective module, then   $\Ext_R^1(\bigoplus\limits_{I\unlhd R}R/I,E)$ is  $u$-$S$-torsion by Theorem \ref{s-inj-ext}. Thus there is an element $s\in S$ such that $s\Ext_R^1(\bigoplus\limits_{I\unlhd R}R/I,E)=s\prod\limits_{I\unlhd R}\Ext_R^1(R/I,E)=0$.  So $s\Ext_R^1(R/I,E)=0$   for any ideal $I$ of $R$.

Suppose  $E$ is an $S$-divisible $R$-module. Let $B$ be an $R$-module, $A$ a submodule of $B$ and $s$  an element in  $S$ satisfying the necessity. Let $f:A\rightarrow E$ be an $R$-homomorphism. Set
\begin{center}
 $\Gamma=\{(C,d)|C$ is a submodule of $B$ containing $A$ and $d|_A=sf\}.$
\end{center}
Since $(A,sf)\in \Gamma$, $\Gamma$ is nonempty. Set $(C_1,d_1)\leq (C_2,d_2)$ if and only if $C_1\subseteq C_2$ and $d_2|_{C_1}=d_1$. Then $\Gamma$ is a partial order. For any chain $(C_j,d_j)$, let $C_0=\bigcup\limits_{j}C_j$ and $d_0(c)=d_j(c)$ if $c\in C_j$. Then $(C_0,d_0)$ is the upper bound of the chain $(C_j,d_j)$. By Zorn's Lemma,  there is a maximal element $(C,d)$  in $\Gamma$.

We claim that $C=B$. On the contrary, let $x\in B-C$. Denote $I=\{r\in R|rx\in C\}$. Then $I$ is an ideal of $R$. Since $E=sE$, there exists a homomorphism $h:I\rightarrow E$ satisfy that $sh(r)=d(rx)$. Then there is an $R$-homomorphism $g: R\rightarrow E$ such that $g(r)=sh(r)=d(rx)$ for any $r\in I$. Let $C_1=C+Rx$ and $d_1(c+rx)=d(c)+g(r)$ where $c\in C$ and $r\in R$. If  $c+rx=0$, then $r\in I$ and thus $d(c)+g(r)=d(c)+sh(r)=d(c)+d(rx)=d(c+rx)=0$. Hence $d_1$ is a well-defined homomorphism such that $d_1|_A=sf$.  So $(C_1,d_1)\in  \Gamma$. However, $(C_1,d_1)> (C,d)$ which contradicts the maximality of $(C,d)$.
\end{proof}

 Now, we give the main result of this section.

\begin{theorem}\label{s-injective-ext} {\bf (Cartan-Eilenberg-Bass Theorem for uniformly $S$-Noetherian rings)}
Let $R$ be a ring, $S$ a regular multiplicative subset of $R$. Then the following assertions are equivalent:
\begin{enumerate}
\item  $R$ is  $u$-$S$-Noetherian;
\item  any direct sum of injective modules is $u$-$S$-injective;
\item  any direct union of injective modules is $u$-$S$-injective.
\end{enumerate}
\end{theorem}

\begin{proof} $(1)\Rightarrow(3):$ Let $\{E_i,f_{i,j}\}_{i<j\in \Lambda}$ be a direct system of injective modules where each $f_{i,j}$ is the embedding map. Let $\lim\limits_{\longrightarrow}{E_i}$ be its direct limit. Let $s$ be an element in $S$ such that for any ideal $I$ of $R$ there exists a finitely generated sub-ideal $K$ of $I$ such that $sI\subseteq K$. Considering the short exact sequence $0\rightarrow I/K\rightarrow R/K\rightarrow R/I\rightarrow 0$, we have the following long exact sequence: $$\Hom_R(I/K,\lim\limits_{\longrightarrow}{E_i})\rightarrow \Ext_R^1(R/I,\lim\limits_{\longrightarrow}{E_i})\rightarrow \Ext_R^1(R/K,\lim\limits_{\longrightarrow}{E_i})\rightarrow \Ext_R^1(I/K,\lim\limits_{\longrightarrow}{E_i}).$$ Since $R/K$ is finitely presented, we have $\Ext_R^1(R/K,\lim\limits_{\longrightarrow}{E_i})\cong \lim\limits_{\longrightarrow}\Ext_R^1(R/K,{E_i})=0$ by the Five Lemma and \cite[Theorem 24.10]{w}. By the proof of  Lemma \ref{u-S-tor-ext}, one can show that $s\Hom_R(I/K,\lim\limits_{\longrightarrow}{E_i})=0$. Thus $s\Ext_R^1(R/I,\lim\limits_{\longrightarrow}{E_i})=0$ for any ideal $I$ of $R$. Since $S$ is  composed of non-zero-divisors, each $E_i$ is $S$-divisible by the proof \cite[Theorem 2.4.5]{fk16}. Thus $\lim\limits_{\longrightarrow}{E_i}$ is also $S$-divisible. So $\lim\limits_{\longrightarrow}{E_i}$ is  $u$-$S$-injective by Proposition \ref{s-inj-baer}.

$(3)\Rightarrow(2):$  Trivial.

$(2)\Rightarrow(1):$ Assume $R$ is not a $u$-$S$-Noetherian ring. By Theorem \ref{u-s-noe-char}, for any  $s\in S$, there exists a strictly ascending chain $I_1\subset I_2\subset ...$ of ideals of $R$ such that for any $k\geq 1$ there is $n\geq k$ satisfying $sI_n\not\subseteq I_k$. Set $I=\bigcup\limits_{i=1}^{\infty}I_i$. Then $I$ is an ideal of $R$ and $I/I_i\not=0$ for any $i\geq 1$. Denote by $E(I/I_i)$ the injective envelope of $I/I_i$. Let $f_i$ be the natural composition $I\twoheadrightarrow I/I_i\rightarrowtail E(I/I_i)$. Since $sI_n\not\subseteq I_i$ for any $i\geq 1$ and some $n\geq i$, we have $sf_i\not=0$ for any $i\geq 1$. We define $f:I\rightarrow \bigoplus_{i=1}^{\infty}E(I/I_i)$ by $f(a)=(f_i(a))$. Not that for each $a\in I$, we have $a\in I_i$ for some $i\geq 1$. So $f$ is a well-defined $R$-homomorphism. Let $\pi_i:\bigoplus_{i=1}^{\infty}E(I/I_i)\twoheadrightarrow E(I/I_i)$ be the projection. The embedding map $i: I\rightarrow R$ induces an exact sequence $$\Hom_R(R,\bigoplus_{i=1}^{\infty}E(I/I_i))\xrightarrow{i^{\ast}} \Hom_R(I,\bigoplus_{i=1}^{\infty}E(I/I_i))\xrightarrow{\delta} \Ext_R^1(R/I,\bigoplus_{i=1}^{\infty}E(I/I_i))\rightarrow 0.$$Since $\bigoplus_{i=1}^{\infty}E(I/I_i)$ is $u$-$S$-injective, there is an $s\in S$ such that $$s\Ext_R^1(R/I,\bigoplus_{i=1}^{\infty}E(I/I_i))=0.$$
Thus there exists a homomorphism $g:R\rightarrow \bigoplus_{i=1}^{\infty}E(I/I_i)$ such that $sf=i^{\ast}(g)$.
Thus for sufficiently large $i$, we have $s\pi_if(a)=\pi_ii^{\ast}(g)(a)=a\pi_ii^{\ast}(g)(1)=0$ for any $a\in I$.
So for such $i$, $sf_i=s\pi_if:I\rightarrow E(I/I_i)$ is  a zero homomorphism, which is a contradiction. Hence $R$ is  $u$-$S$-Noetherian.
\end{proof}


\begin{acknowledgement}\quad\\
The fifth author was supported by the National Natural Science Foundation of China (No. 12061001).
\end{acknowledgement}

\end{document}